\documentclass[journal]{IEEEtran}
\usepackage{xcolor}
\usepackage{graphicx}
\usepackage{amsmath,amssymb}
\usepackage{xcolor}
\usepackage{setspace} 
\usepackage{units}
\usepackage[absolute,overlay]{textpos}
\usepackage{multirow}
\usepackage{multicol}
\usepackage{amsthm}
\usepackage{bm}
\usepackage{subfigure}
\usepackage{calligra}
\newtheorem{assum}{Assumption}
\newtheorem{definition}{Definition}

\newtheorem{theorem}{Theorem}

\title{\LARGE \bf
Discrete-Time Event-Triggered Extremum Seeking
}

\author{Victor Hugo Pereira Rodrigues$^{a}$,
Tiago Roux Oliveira$^{a}$,
Miroslav Krsti{\' c}$^{b}$,
Frank Allg{\" o}wer$^{c}$
\thanks{$^{a}$Department of Electronics and Telecommunication Engineering, State University of Rio de Janeiro (UERJ), Rio de Janeiro--RJ, Brazil (\tt\small{victor.rodrigues@uerj.br;  tiagoroux@uerj.br})}
\thanks{$^{b}$Department of Mechanical and Aerospace Engineering, University of California at San Diego (UCSD), La Jolla--CA, USA (\tt\small{mkrstic@ucsd.edu})}
\thanks{$^{c}$Institute for Systems Theory and Automatic Control, University of Stuttgart, Germany (\tt\small{allgower@ist.uni-stuttgart.de})}
}

\begin{document}

\maketitle
\thispagestyle{empty}
\pagestyle{empty}

\begin{abstract}

This paper proposes a discrete-time event-triggered extremum seeking control scheme for real-time optimization of nonlinear systems. Unlike conventional discrete-time implementations relying on periodic updates, the proposed approach updates the control input only when a state-dependent triggering condition is satisfied, reducing unnecessary actuation and communication. The resulting closed-loop system combines extremum seeking with an event-triggering mechanism that adaptively determines the input update instants. Using discrete-time averaging and Lyapunov analysis, we establish practical convergence of the trajectories to a neighborhood of the unknown  extremum point and show exponential stability of the associated average dynamics. The proposed method preserves the optimization capability of classical extremum seeking while significantly reducing the number of input updates. Simulation results illustrate the effectiveness of the approach for resource-aware real-time optimization.

\end{abstract}

\section{Introduction}
\label{sec:introduction}

This paper introduces a discrete-time event-triggered extremum seeking control  strategy for real-time optimization of nonlinear systems. Departing from the conventional paradigm of periodic discrete-time updates, the proposed method integrates model-free  extremum seeking (ES) with a state-dependent triggering rule that determines when the control input should be updated. In this way, the algorithm preserves the optimization capability of classical ES \cite{KW:2000,AK:2003,TRoux:2022} while significantly reducing unnecessary actuation and communication \cite{T:2007}.

It is worth noting that our previous work \cite{VHPR:2023a} on event-triggered extremum seeking was developed in a continuous-time setting, where the triggering mechanism was designed in conjunction with the continuous evolution of the plant and adaptation dynamics. In contrast, the present work addresses the fundamentally different challenges that arise in discrete-time implementations, where sampling, input holding, and update constraints play a central role in the system behavior. Hence, the extension from continuous-time to discrete-time event-triggered extremum seeking is therefore not a straightforward discretization, but constitutes the main original contribution of this paper, requiring a distinct design and analysis framework tailored to discrete-time dynamics.

From a theoretical standpoint, the analysis combines discrete-time averaging \cite{BFS:1988,PPY:2004} with Lyapunov arguments to show practical convergence of the closed-loop trajectories to a neighborhood of the unknown extremum and exponential stability of the associated average dynamics. These results provide a rigorous justification that the introduction of the triggering mechanism does not compromise the fundamental ES behavior, but rather complements it with a resource-aware update policy.

Beyond the specific design proposed here, this work highlights a broader and still underexplored direction: event-triggered control in discrete time. While most of the event-triggered literature has been developed in continuous time \cite{EDK:2010}, discrete-time formulations are particularly natural for digital implementations, where sampling, computation, and communication constraints inherently shape how and when updates can occur \cite{frank:2023,frank:2023b}. In this sense, the proposed scheme is not only an ES contribution, but also a step toward bridging discrete-time control and event-triggered mechanisms in a practically meaningful way.

Simulation results illustrates that a substantial reduction in input updates can be achieved without sacrificing optimization performance, reinforcing the relevance of the approach for model-free resource-constrained real-time applications.

\section{Preliminaries}
\label{sec:preliminaries}

\textbf{Notation.} Throughout the manuscript, the 2-norm (Euclidean) of vectors and induced norm of matrices are denoted by double bars $\|\cdot\|$ while absolute value of scalar variables are denoted by single bars $|\cdot|$. The terms $\lambda_{\min}(\cdot)$ and $\lambda_{\max}(\cdot)$ denote the minimum and maximum eigenvalues of a given positive definite matrix, respectively. Consider $\varepsilon \in \lbrack -\varepsilon_{0}\,, \varepsilon_{0} \rbrack \subset \mathbb{R}$ and the mappings $\delta_{1}(\varepsilon)$ and $\delta_{2}(\varepsilon)$, where $\delta_{1}: \lbrack -\varepsilon_{0}\,, \varepsilon_{0} \rbrack \to \mathbb{R}$ and $\delta_{2}: \lbrack -\varepsilon_{0}\,, \varepsilon_{0} \rbrack \to \mathbb{R}$. The function $\delta_1(\varepsilon)$ has magnitude of order $\delta_2(\varepsilon)$, denoted by $\delta_{1}(\varepsilon) = \mathcal{O}(\delta_{2}(\varepsilon))$, if there exist positive constants $k$ and $c$ such that $|\delta_{1}(\varepsilon)| \leq k |\delta_{2}(\varepsilon)|$, for all $|\varepsilon|<c$.

\begin{definition}[Shift Operator, \cite{AW:1997}] \label{def:shiftOperator}

Let $f=\{f[k]\}_{k\in\mathbb{Z}}$ denote a discrete-time sequence, where $f[k] \in \mathbb{R}^{n}$ represents the value of the sequence at index $k \in \mathbb{Z}$. Assume that $f[k]$ is obtained by sampling a continuous-time signal $f_a(t) \in \mathbb{R}^{n}$ with sampling period $\epsilon>0$, {\it i.e.}, $f[k] = f_a(k\epsilon)$, $k\in\mathbb{Z}$. The {\bf forward shift operator} $q$ is defined by $qf[k] = f[k+1]$. Similarly, the {\bf backward shift operator} is $q^{-1}f[k] = f[k-1]$. In the $z$-domain, $q$ is identified with the complex variable $z \in \mathbb{C}$.
\end{definition}

\begin{figure*}[h!]
\centering
\includegraphics[width=13cm]{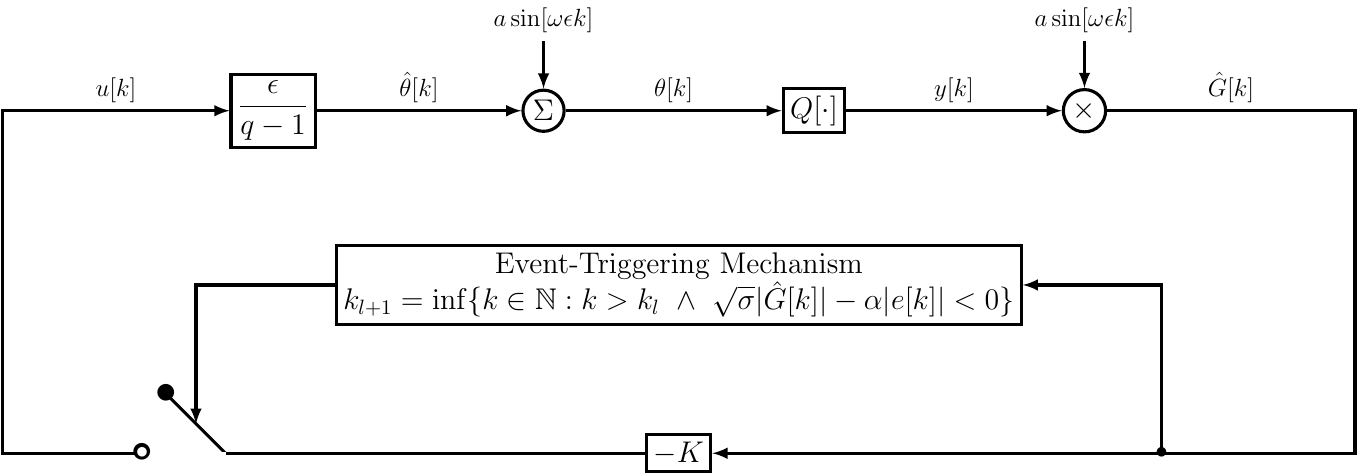}
\caption{Event-Triggered for Discrete-Time Gradient-based Extremum Seeking Control.}
\label{fig:BD_GradientES_v2}
\end{figure*}

\section{Problem Formulation} \label{sec:prblFrm_siso}

We define the following  nonlinear static map
\begin{align}
Q[\theta[k]] = Q^{\ast}+\frac{H^{\ast}}{2}(\theta[k]-\theta^{\ast})^{2}\,, \label{eq:Q_1_event_siso}
\end{align}
where $Q^{\ast}\in \mathbb{R}$ is the unknown extremum, $H^{\ast} \in \mathbb{R}-\{0\}$ is the unknown Hessian, $\theta^{\ast} \in \mathbb{R}$ is the unknown optimizer, and the input of the map $\theta[k]\in \mathbb{R}$ is designed as the real-time estimate $\hat{\theta}[k]\in \mathbb{R}$ of $\theta^{\ast}$ additively perturbed by the sinusoid $a \sin[\omega \epsilon k]$, {\it i.e.},
\begin{align}
\theta[k]=\hat{\theta}[k]+a \sin[\omega \epsilon k]\,. \label{eq:theta_event_siso}
\end{align}
The Hessian $H^{\ast}\!<\!0$ for maximization problems, and $H^{\ast}\!>\!0$ for minimization purposes.     

Let us define the estimation error 
\begin{align}
\tilde{\theta}[k]=\hat{\theta}[k]-\theta^{\ast}\,, \label{eq:thetaTilde_event_siso}
\end{align}
and the gradient estimate
\begin{align}
\hat{G}[k]=a \sin[\omega \epsilon k]~y[k]\,, \label{eq:hatG_event_siso}
\end{align}
by the employing the demodulation signal, $a \sin[\omega \epsilon k]$, which has nonzero amplitudes $a$ and frequency $\omega$ \cite{GKN:2012,K:2014}.

From (\ref{eq:theta_event_siso}) and (\ref{eq:thetaTilde_event_siso}), we can write
\begin{align}
\theta[k]=\tilde{\theta}[k]+a \sin[\omega \epsilon k]+\theta^{\ast}\,, \label{eq:theta_2_event_siso}
\end{align}
and, therefore, by using (\ref{eq:theta_2_event_siso}) the output $y[k]=Q[\theta[k]]$, with $Q[\theta[k]]$ given by (\ref{eq:theta_event_siso}) can also be written as
\begin{align}
y(t)&=Q^{\ast}+\frac{H^{\ast}a^{2}}{4}+\frac{H^{\ast}}{2}\tilde{\theta}^{2}[k]+a \sin[\omega \epsilon k]H^{\ast}\tilde{\theta}[k]\nonumber\\
&\quad-\frac{H^{\ast}a^{2}}{4} \cos[2\omega \epsilon k]\,.
 \label{eq:y_1_event_siso}
\end{align}
Thus, from (\ref{eq:hatG_event_siso}) and (\ref{eq:y_1_event_siso}), the gradient estimate, is given by
\begin{align}
\hat{G}[k]&= \frac{a^2 H^{\ast}}{2}\left(1-\cos[2\omega \epsilon k]\right)\tilde{\theta}[k]+\frac{a H^{\ast}}{2}\sin[\omega \epsilon k]\tilde{\theta}^{2}[k]\nonumber \\
&\quad+\left(a Q^{\ast}+\frac{3 a^{3}H^{\ast}}{8}\right)\sin[\omega \epsilon k]-\frac{a^{3}H^{\ast}}{8}\sin[3\omega \epsilon k]\,. \label{eq:hatG_2_event_siso}
\end{align} 
Notice the quadratic term in $\tilde{\theta}[k]$ in (\ref{eq:hatG_2_event_siso}) may be neglected in a local analysis \cite{AK:2003}. Thus, hereafter the gradient estimate is given by
\begin{align}
\hat{G}[k]&=\frac{a^2 H^{\ast}}{2}\left(1-\cos[2\omega \epsilon k]\right)\tilde{\theta}[k]\nonumber \\
&\quad+\left(a Q^{\ast}+\frac{3 a^{3}H^{\ast}}{8}\right)\sin[\omega \epsilon k]-\frac{a^{3}H^{\ast}}{8}\sin[3 \omega \epsilon k] \,. \label{eq:hatG_4_event_siso}
\end{align}
On the other hand, from the \textit{Event-Triggered for Discrete-Time Gradient-based Extremum Seeking Control} scheme depicted in Fig.~\ref{fig:BD_GradientES_v2}, by using the shift operator given in Definition~\ref{def:shiftOperator}, the discrete estimate of $\theta^{\ast}$ can be found as 
\begin{align}
    \hat{\theta}[k]=\frac{h}{q-1}u[k]\,. \label{eq:thetaHat_20260313}
\end{align}
Then, by plugging (\ref{eq:thetaHat_20260313}) in (\ref{eq:thetaTilde_event_siso}), the dynamics that governs $\tilde{\theta}[k]$, is given by
\begin{align}
\tilde{\theta}[k]&=\frac{\epsilon}{q-1}u[k]-\theta^{\ast} \label{eq:dtildeThetadt_20250206_1}\,, 
\end{align}
or, equivalently, 
\begin{align}
\tilde{\theta}[k+1]&=\tilde{\theta}[k]+\epsilon u[k] \label{eq:tildeThetaK+1_20260313_1}\,, 
\end{align}
with $u[k] \in \mathbb{R}$ since $\theta^{\ast}[k]=\theta^{\ast}$, such that $(q-1)\theta^{\ast}[k]=\theta^{\ast}[k+1]-\theta^{\ast}[k]=\theta^{\ast}-\theta^{\ast}=0$. Moreover, by using (\ref{eq:hatG_4_event_siso}), the gradient estimate in the iteration $k+1$ is given by
\begin{align}
&\hat{G}[k+1]=
\nonumber \\
&\frac{a^2 H^{\ast}}{2}\left(1\mathbb{-}\cos[2\omega \epsilon]\cos[2\omega \epsilon k]\mathbb{+}\sin[2\omega \epsilon]\sin[2\omega \epsilon k]\right)\tilde{\theta}[k\mathbb{+}1]\nonumber \\
&+\left(a Q^{\ast}+\frac{3 a^{3}H^{\ast}}{8}\right)\left(\cos[\omega \epsilon]\sin[\omega \epsilon k]+\sin[\omega \epsilon]\cos[\omega \epsilon k]\right) \nonumber \\
&-\frac{a^{3}H^{\ast}}{8}\left(\cos[3\omega \epsilon]\sin[3\omega \epsilon k]+\sin[3\omega \epsilon]\cos[3\omega \epsilon k]\right) \,, \label{eq:hatG_20260313_2} 
\end{align}
thus, plugging (\ref{eq:tildeThetaK+1_20260313_1}) in (\ref{eq:hatG_20260313_2}), one has 
\begin{align}
&\hat{G}[k+1]=
\nonumber \\
&\frac{a^2 H^{\ast}}{2}\left(1\mathbb{-}\cos[2\omega \epsilon]\cos[2\omega \epsilon k]\mathbb{+}\sin[2\omega \epsilon]\sin[2\omega \epsilon k]\right)\tilde{\theta}[k]\nonumber \\
&\mathbb{+}\frac{a^2 H^{\ast}}{2}\left(1\mathbb{-}\cos[2\omega \epsilon]\cos[2\omega \epsilon k]\mathbb{+}\sin[2\omega \epsilon]\sin[2\omega \epsilon k]\right)\epsilon u[k]\nonumber \\
&\mathbb{+}\left(a Q^{\ast}+\frac{3 a^{3}H^{\ast}}{8}\right)\left(\cos[\omega \epsilon]\sin[\omega \epsilon k]+\sin[\omega \epsilon]\cos[\omega \epsilon k]\right) \nonumber \\
&\mathbb{-}\frac{a^{3}H^{\ast}}{8}\left(\cos[3\omega \epsilon]\sin[3\omega \epsilon k]+\sin[3\omega \epsilon]\cos[3\omega \epsilon k]\right) \,. \label{eq:hatG_20260329_1} 
\end{align}
Now, \textcolor{blue}{adding} and \textcolor{red}{subtracting} the terms $\cos[2\omega \epsilon k]\tilde{\theta}[k]$, $\left(a Q^{\ast}+\frac{3 a^{3}H^{\ast}}{8}\right)\sin[\omega \epsilon k]$, and $\frac{a^{3}H^{\ast}}{8}\sin[3 \omega \epsilon k]$, equation (\ref{eq:hatG_20260329_1}) can be  rewritten as
\begin{small} 
\begin{align}
&\hat{G}[k+1]=\frac{a^2 H^{\ast}}{2}\left(1\textcolor{red}{\mathbb{-}\cos[2\omega \epsilon k]}\right)\tilde{\theta}[k]
\nonumber \\
&\textcolor{blue}{+\left(a Q^{\ast}+\frac{3 a^{3}H^{\ast}}{8}\right)\sin[\omega \epsilon k]} \textcolor{red}{-\frac{a^{3}H^{\ast}}{8}\sin[3 \omega \epsilon k]} \nonumber \\
&\frac{a^2 H^{\ast}}{2}\left(\textcolor{blue}{\cos[2\omega \epsilon k]}\mathbb{-}\cos[2\omega \epsilon]\cos[2\omega \epsilon k]\mathbb{+}\sin[2\omega \epsilon]\sin[2\omega \epsilon k]\right)\tilde{\theta}[k]\nonumber \\
&\mathbb{+}\frac{a^2 H^{\ast}}{2}\left(1\mathbb{-}\cos[2\omega \epsilon]\cos[2\omega \epsilon k]\mathbb{+}\sin[2\omega \epsilon]\sin[2\omega \epsilon k]\right)\epsilon u[k]\nonumber \\
&\mathbb{+}\!\left(\! a Q^{\ast}\mathbb{+}\frac{3 a^{3}H^{\ast}}{8} \!\right)\!\left(\cos[\omega \epsilon]\sin[\omega \epsilon k]\mathbb{+}\sin[\omega \epsilon]\cos[\omega \epsilon k]\mathbb{-}\textcolor{red}{\sin[\omega \epsilon k]}\right) \nonumber \\
&\mathbb{-}\frac{a^{3}H^{\ast}}{8}\left(\cos[3\omega \epsilon]\sin[3\omega \epsilon k]+\sin[3\omega \epsilon]\cos[3\omega \epsilon k]\mathbb{-}\textcolor{blue}{\sin[3\omega \epsilon k]}\right) \,. \label{eq:hatG_20260329_1_} 
\end{align}
\end{small}
Then, defining
\begin{align}
\Delta H^{\ast}[k]&\mathbb{=}-H^{\ast}\cos[2 \omega \epsilon k] \,, \label{eq:DeltaHast} \\
\Delta \tilde{H}^{\ast}[k]&\mathbb{=}\Delta H^{\ast}[k+1]-\Delta H^{\ast}[k] \\
&\mathbb{=}-H^{\ast}\cos[2\omega \epsilon]\cos[2\omega \epsilon k]+H^{\ast}\sin[2\omega \epsilon]\sin[2\omega \epsilon k] \nonumber \\
&\quad+H^{\ast}\cos[2\omega \epsilon k]\,, \label{eq:DeltaTildeHast} \\
\delta[k]&\mathbb{=}\left(a Q^{\ast}+\frac{3 a^{3}H^{\ast}}{8}\right)\sin[\omega \epsilon k]-\frac{a^{3}H^{\ast}}{8}\sin[3 \omega \epsilon k] \,,  \label{eq:delta} \\
\tilde{\delta }[k]&\mathbb{=}\delta [k+1]-\delta [k] \nonumber  \\
&=\!\left(\! a Q^{\ast}\mathbb{+}\frac{3 a^{3}H^{\ast}}{8} \!\right)\!\left(\cos[\omega \epsilon]\sin[\omega \epsilon k] \right. \nonumber \\
&\quad\left.\mathbb{+}\sin[\omega \epsilon]\cos[\omega \epsilon k]\mathbb{-}\sin[\omega \epsilon k]\right) \nonumber \\
&\quad\mathbb{-}\frac{a^{3}H^{\ast}}{8}\left(\cos[3\omega \epsilon]\sin[3\omega \epsilon k] \right. \nonumber \\
&\quad \left.+\sin[3\omega \epsilon]\cos[3\omega \epsilon k]\mathbb{-}\sin[3\omega \epsilon k]\right)\,, \label{eq:tilde_delta}
\end{align}
we can use (\ref{eq:DeltaHast})--(\ref{eq:tilde_delta}) to rewrite (\ref{eq:hatG_4_event_siso}) and (\ref{eq:hatG_20260329_1_}), respectively, as 
\begin{align}
\hat{G}[k]&=\frac{a^2}{2}\left(H^{\ast}+\Delta H^{\ast}[k]\right)\tilde{\theta}[k]+\delta[k] \,,\label{eq:hatG_20260313_1}  \\
\hat{G}[k+1]&=\hat{G}[k] +\frac{a^2}{2}\Delta \tilde{H}^{\ast}[k]\tilde{\theta}[k] \nonumber \\
&\quad+\frac{a^2}{2}\left(H^{\ast}+\Delta H^{\ast}[k]+\Delta \tilde{H}^{\ast}[k]\right)\epsilon u[k]+\tilde{\delta}[k]\,. \label{eq:hatG_20260313_4} 
\end{align}

\subsection{Event-Triggered of the Discrete-Time Gradient-Based\\ Extremum Seeking Tuning Law}

We consider the following static gradient-based feedback tuning law, given by 
\begin{align}
u[k]=-K\hat{G}[k] \,, \quad \forall k \in \mathbb{Z} \label{eq:u}
\end{align}
with the gain $K$ being designed such that $\text{sign}(K)=\text{sign}(H^*)$ and $\left|1-\frac{a^2\epsilon H^{\ast}K}{2}\right|\in (0,1)$. We assume that the tuning law (\ref{eq:u}) stabilizes the system at the corresponding average equilibrium $\hat{G}_{\rm{av}} \equiv 0$ with local exponential convergence.

Assuming no delays in the sensor-to-controller and controller-to-actuator communication paths, the control input is updated only at the discrete triggering iteration step $\{k_{l}\}_{l\in \mathbb{N}}$ and held constant over the inter-execution interval. In particular, the control input computed at the last triggering instant $k_l$ is applied for all $k \in [k_l, k_{l+1})$. Thus, the tuning law is given by
\begin{align}
u[k] = -K\hat{G}[k_l]\,, \quad \forall k \in [k_l, k_{l+1})\,, \quad l \in \mathbb{N}. 
\label{eq:u_discrete_event}
\end{align}
Since the control updates occur only at the triggering instants, the state-dependent quantities evolve between updates while the control input remains constant. This motivates the definition of the measurement error over the inter-execution interval as
\begin{align}
e[k] := \hat{G}[k_i] - \hat{G}[k]\,, \quad \forall k \in [k_l, k_{l+1})\,, \quad l \in \mathbb{N}. 
\label{eq:e_discrete_event}
\end{align}

Therefore, by using (\ref{eq:hatG_20260313_1}) and (\ref{eq:e_discrete_event}), $\forall k \in [k_l, k_{l+1})$, $l \in \mathbb{N}$, the event-triggered for discrete-time gradient-based tuning law (\ref{eq:u_discrete_event}) can be rewritten as
\begin{align}
u[k] &= -K\hat{G}[k]-Ke[k]  \,, \label{eq:u_20250602_v2} \\
&=-K\frac{a^2}{2}\left(H^{\ast}+\Delta H^{\ast}[k]\right)\tilde{\theta}[k]-K\delta[k]-Ke[k] \,. \label{eq:u_20250602_v3}
\end{align}
Now, plugging (\ref{eq:u_20250602_v2}) into (\ref{eq:hatG_20260313_4}) and (\ref{eq:u_20250602_v3}) into (\ref{eq:tildeThetaK+1_20260313_1}), we arrive at the following Input-to-State Stable (ISS) \cite{K:2002} representations for the dynamics of $\hat{G}[k]$ and $\tilde{\theta}[k]$ with respect to the error vector $e[k]$ in (\ref{eq:e_discrete_event}), $\forall k \in [k_l, k_{l+1})$, $l \in \mathbb{N}$: 
\begin{align}
\hat{G}[k\mathbb{+}1]&\mathbb{=}\hat{G}[k] +\frac{a^{2}}{2}\Delta \tilde{H}^{\ast}[k]\tilde{\theta}[k] \nonumber \\
&\quad-\epsilon\frac{a^{2}}{2}\left(H^{\ast}+\Delta H^{\ast}[k]+\Delta \tilde{H}^{\ast}[k]\right)K\hat{G}[k] \nonumber \\
&\quad-\epsilon\frac{a^{2}}{2}\left(H^{\ast}+\Delta H^{\ast}[k]+\Delta \tilde{H}^{\ast}[k]\right)Ke[k]+\tilde{\delta}[k]\,, \label{eq:hatG_20260318_1} \\
\tilde{\theta}[k\mathbb{+}1]&\mathbb{=}\tilde{\theta}[k]\mathbb{-}\epsilon K\frac{a^{2}}{2}\left(H^{\ast}\mathbb{+}\Delta H^{\ast}[k]\right)\tilde{\theta}[k]\mathbb{-}\epsilon Ke[k]\mathbb{-}\epsilon K\delta[k]\,. \label{eq:tildeThetaK+1_20260318_1}
\end{align}

\subsection{Event-Triggering Mechanism} 

In what follows, Definition~\ref{def:staticEvent_discrete} illustrates how the small design parameter $\sigma \in (0,1)$ and the error signal $e[k]$—representing the deviation between the gradient estimate $\hat{G}[k]$ and its last transmitted value—are leveraged to construct a static event-triggered mechanism for scalar 
gradient-based extremum seeking in discrete-time. The resulting strategy triggers the recalculation of the tuning law, as defined in (\ref{eq:u_discrete_event}) only when needed. The closed-loop system preserves asymptotic stability and guaranteed convergence, as supported by the theoretical results in \cite{EDK:2010}.

Before stating Definition~\ref{def:staticEvent_discrete}, we introduce the following assumption which will be used in the sequence.
\begin{assum} \label{assumption_lyapunovEq}
 There exist $\epsilon$, $a$, and $K$ such that  
 \begin{align}
      0<\left|1 -\frac{\epsilon a^{2} H^{\ast}K}{2}\right| <1\,,
 \end{align}
 and a known positive constant $\alpha$ satisfying 
\begin{small}
   \begin{align}
    \alpha > \frac{\epsilon a^{2} |H^{\ast}||K|}{\sqrt{2}}\frac{\sqrt{1+7\left(1 -\frac{\epsilon a^{2} H^{\ast}K}{2}\right)^{2}}}{\left(1-\left(1 -\frac{\epsilon a^{2} H^{\ast}K}{2}\right)^{2}\right)} \,. \label{eq:alpha}
   \end{align} 
\end{small}

\end{assum}

\begin{definition}[\small{\textbf{Event-Triggering Condition}}] \label{def:staticEvent_discrete}
The triggering mechanism for discrete-time gradient-based extremum seeking controller with static event-triggered condition consists of two components:
\begin{enumerate}
	\item A sequence of increasing triggering iteration step $\mathcal{K}$ such that $\mathcal{K}=\{k_{0}, k_{1}, k_{2}, \ldots\}$ with $k_{0}=0$, generated according to: \\
		\!\!\!\! $\bullet$ If $\left\{k \in \mathbb{Z}:~ k > k_{l} ~ \wedge ~ \sqrt{\sigma}|\hat{G}[k]| - \alpha |e[k]| < 0 \right\} \mathbb{=} \emptyset$,
		then the set of triggering instants is $\mathcal{K}=\{k_{0}, k_{1}, \ldots, k_{l}\}$. \\
		
		$\bullet$ If $\left\{k \in \mathbb{Z}:~ k > k_{l} ~ \wedge ~ \sqrt{\sigma}|\hat{G}[k]| - \alpha |e[k]| < 0 \right\} \mathbb{\neq} \emptyset$, 
		the next triggering instant is given by
		\begin{align}
		k_{l+1} = \inf\left\{k \in \mathbb{Z}:~ k \mathbb{>} k_{l} ~ \mathbb{\wedge} ~ \sqrt{\sigma}|\hat{G}[k]| \mathbb{-} \alpha |e[k]| \mathbb{<} 0 \right\}.
		\label{eq:k_next_event_}
		\end{align}
	
	\item A feedback control action (\ref{eq:u_discrete_event}) updated at the triggering instants $k_l$ and held constant for all $k \in [k_l, k_{l+1})$.
\end{enumerate}
\end{definition}

\subsection{Average Closed-Loop System}

Defining $\epsilon=h$ and the augmented state as follows
\begin{align}
x[k]=\begin{bmatrix} x_{1}[k]\,, x_{2}[k]\end{bmatrix}^{\top}:=\begin{bmatrix} \hat{G}[k]\,, \tilde{\theta}[k]\end{bmatrix}^{\top}\,, \label{eq:X}
\end{align}
the system (\ref{eq:hatG_20260318_1})--(\ref{eq:tildeThetaK+1_20260318_1}) reduces to
\begin{align}
x[k+1]&=x[k]+\epsilon f[k,x,\epsilon]\,, \label{eq:x_k+1_event}
\end{align}
where
\begin{align}
&f_{1}[k,x,\epsilon]=\mathbb{-}\frac{a^2}{2}\left(H^{\ast}\mathbb{+}\Delta H^{\ast}[k]\mathbb{+}\Delta \tilde{H}^{\ast}[k]\right)Kx_{1}[k] \nonumber \\
& \mathbb{+}\frac{a^2}{2}\Delta \tilde{H}^{\ast}[k]x_{2}[k]\mathbb{-}\frac{a^2}{2}\left(H^{\ast}\mathbb{+}\Delta H^{\ast}[k]\mathbb{+}\Delta \tilde{H}^{\ast}[k]\right)Ke[k]\nonumber \\
&\mathbb{+}\frac{1}{\epsilon}\tilde{\delta}[k]\,, \label{eq:f1_20260321_1} \\
&f_{2}[k,x,\epsilon]=\mathbb{-}K\frac{a^2}{2}\left(H^{\ast}\mathbb{+}\Delta H^{\ast}[k]\right)x_{2}[k]\mathbb{-}Ke[k] \mathbb{-}K\delta[k]\,. \label{eq:f2_20260321_1}
\end{align}
The discrete-time system (\ref{eq:x_k+1_event}) has a small parameter $\epsilon =h$ and the discrete-time vector field $f[k,x,\epsilon]$ is $T$-periodic in $k$ such that 
\begin{align}
T&= \frac{2\pi}{\omega}\,. \label{eq:T}
\end{align}  
Thus, the discrete-time averaging method \cite{BFS:1988,PPY:2004} can be applied to $f[k,x,\epsilon]$ in the limit $\epsilon \to 0$.

The averaging approach allows one to characterize in what sense the behavior of a constructed averaged autonomous system approximates that of the original non-autonomous system (\ref{eq:x_k+1_event}). Intuitively, when the system evolves on a slower time scale than the excitation, its behavior is predominantly governed by the average effect of the excitation over one period.

By applying the discrete-time averaging technique to (\ref{eq:x_k+1_event}), we obtain the averaged system
\begin{align}
x^{\rm{av}}[k+1] &= x^{\rm{av}}[k] + \epsilon f^{\rm{av}}[x^{\rm{av}}], \label{eq:xav_k+1_event} \\
f^{\rm{av}}[x^{\rm{av}}] &= \lim_{\epsilon \to 0}\lim_{T \to \infty} \frac{1}{T} \sum_{k=s+1}^{s+T} f[k,x^{\rm{av}},\epsilon]. \label{eq:fav_event}
\end{align}

Therefore, ``freezing'' the average states variables of $x_{1}[k]=\hat{G}[k]$ and $x_{2}[k]=\tilde{\theta}[k]$ in (\ref{eq:x_k+1_event})--(\ref{eq:f2_20260321_1}), (\ref{eq:hatG_20260313_1}) and (\ref{eq:e_discrete_event}), one gets, for all $k \in [k_l, k_{l+1})$, the corresponding average system
\begin{align}
\hat{G}_{\rm{av}}[k+1]&=\left(1 \!-\!\frac{\epsilon a^{2} H^{\ast}K}{2}\right)\hat{G}_{\rm{av}}[k]\!-\!\frac{\epsilon a^{2}H^{\ast}K}{2}e_{\rm{av}}[k]\,, \label{eq:hatG_av_k+1_1} \\
\tilde{\theta}_{\rm{av}}[k+1]&=\left(1 \!-\!\frac{\epsilon a^{2} KH^{\ast}}{2}\right)\tilde{\theta}_{\rm{av}}[k]\!-\!\epsilon Ke_{\rm{av}}[k]\,, \label{eq:tildeTheta_av_k+1_1} \\
\hat{G}_{\rm{av}}[k]&=H^{\ast}\tilde{\theta}_{\rm{av}}[k]\,, \label{eq:hatG_av_k_1} \\
e_{\rm{av}}[k] &= \hat{G}_{\rm{av}}[k_l] \!-\! \hat{G}_{\rm{av}}[k]\,, \quad \forall k \!\in\! [k_l, k_{l+1})\,, \quad l \!\in\! \mathbb{N}. 
\label{eq:e_av_k_1}
\end{align}

Now, we introduce {\bf Definition~\ref{def:staticEvent_discrete_av}} as an average version of  {\bf Definition~\ref{def:staticEvent_discrete}}.

\begin{definition}[\small{\textbf{Average Event-Triggering Condition}}] \label{def:staticEvent_discrete_av}
The \textit{\textbf{average}} triggering mechanism for discrete-time gradient-based extremum seeking controller with static event-triggered condition and the small design parameter $\sigma \in (0,1)$ consists of two components:
\begin{enumerate}
	\item A sequence of increasing triggering iteration step $\mathcal{K}$ such that $\mathcal{K}=\{k_{0}, k_{1}, k_{2}, \ldots\}$ with $k_{0}=0$, generated according to: \\
		\!\!\!\! $\bullet$ If  $\left\{k \in \mathbb{Z}:~ k > k_{l} ~ \mathbb{\wedge} ~ \sqrt{\sigma}|\hat{G}_{\rm{av}}[k]| - \alpha |e_{\rm{av}}[k]| \mathbb{<} 0 \right\} \mathbb{=} \emptyset$,
		then the set of triggering instants is $\mathcal{K}=\{k_{0}, k_{1}, \ldots, k_{l}\}$. \\
		
		$\bullet$ If $\left\{k \in \mathbb{Z}:~ k > k_{l} ~ \mathbb{\wedge} ~ \sqrt{\sigma}|\hat{G}_{\rm{av}}[k]| - \alpha |e_{\rm{av}}[k]| \mathbb{<} 0 \right\} \mathbb{\neq} \emptyset$,
		the next triggering instant is given by
		\begin{align}
		k_{l\mathbb{+}1} \mathbb{=} \inf\left\{k \mathbb{\in} \mathbb{Z}:~ k \mathbb{>} k_{l} ~ \mathbb{\wedge} ~ \sqrt{\sigma}|\hat{G}_{\rm{av}}[k]| \mathbb{-} \alpha |e_{\rm{av}}[k]| \mathbb{<} 0 \right\}.
		\label{eq:k_next_event}
		\end{align}
	
	\item A feedback control action (\ref{eq:u_discrete_event}) updated at the triggering instants $k_l$ and held constant for all $k \in [k_l, k_{l+1})$,
    		\begin{align}
			u_{\rm{av}}[k]=-K\hat{G}_{\rm{av}}[k_{l}] \,. \label{eq:U_MD2}
		\end{align}
\end{enumerate}
\end{definition}

\section{Stability Analysis}\label{sec:ET-DT-gradient-ES_stability}

Theorem~\ref{thm:NETESC_2} states the local exponential practical stability of the ES of Fig.~\ref{fig:BD_GradientES_v2} based on the proposed discrete-time event-triggering execution mechanism.

\begin{theorem} \label{thm:NETESC_2}
Consider the closed-loop average dynamics of the gradient estimate (\ref{eq:hatG_av_k+1_1}), the average error vector (\ref{eq:e_av_k_1}), under  Assumption \ref{assumption_lyapunovEq}, and the average \textbf{static} event-triggering mechanism given by \textbf{Definition \ref{def:staticEvent_discrete_av}}. For $\epsilon>0$ sufficiently small, the equilibrium $(\hat{G}_{\rm{av}}\,,\tilde{\theta}_{\rm{av}})=(0\,,0)$ is locally exponentially stable such that the following inequalities can be obtained for the non-average signals:
\begin{align}
|\theta[k]\mathbb{-}\theta^{\ast}|  &\leq  \left(1\mathbb{-} \left(1\mathbb{-}\left(1 \mathbb{-}\frac{\epsilon a^{2} H^{\ast}K}{2}\right)^{2}\right)\frac{\left(1 \mathbb{-} \sigma\right)}{2}\right)^{\frac{k}{2}} \nonumber \\
 &\quad\times|\theta[0]-\theta^{\ast}|+\mathcal{O}\left(a+\epsilon\right)\,,  \label{eq:normTheta_thm1} \\
   |y[k] - Q^{\ast}|&\leq 2\left(1\mathbb{-} \left(1\mathbb{-}\left(1 \mathbb{-}\frac{\epsilon a^{2} H^{\ast}K}{2}\right)^{2}\right)\frac{\left(1 \mathbb{-} \sigma\right)}{2}\right)^{k} \nonumber \\
 &\quad\mathbb{\times}|y[0]-Q^{\ast}|+\mathcal{O}\left(a^2+\epsilon^2\right)\,. \label{eq:normY_thm1} 
\end{align}
In addition, there exists a lower bound  $k^{\ast}$ for the inter-execution interval $k_{l+1}-k_{l}$, for all $l \in \mathbb{N}$, precluding the Zeno-like behavior.
\end{theorem} 

\begin{proof}
The proof of the theorem is divided into two parts: practical asymptotic stability and avoidance of Zeno-like behavior.

\begin{flushleft}
\underline{\it A. Practical Asymptotic Stability}
\end{flushleft}

Now, consider the following Lyapunov function candidate for the average system (\ref{eq:hatG_av_k+1_1}): 
\begin{align}
    V[\hat{G}_{\rm{av}}[k]]=\hat{G}_{\rm{av}}^{2}[k]\,. \label{eq:lyapfunc}
\end{align}
Along the trajectories of the system, by using (\ref{eq:hatG_av_k+1_1}), we can write 
\begin{align}
    &\Delta V[\hat{G}_{\rm{av}}[k]]= V[\hat{G}_{\rm{av}}[k+1]]-V[\hat{G}_{\rm{av}}[k]] \label{eq:lyapfunc_0} \\
    &=\hat{G}_{\rm{av}}^{2}[k+1]-\hat{G}_{\rm{av}}^{2}[k] \nonumber \\
    &= \left(1 -\frac{\epsilon a^{2} H^{\ast}K}{2}\right)^{2}\hat{G}_{\rm{av}}^{2}[k]+\frac{\epsilon^{2} a^{4} H^{\ast 2}K^{2}}{4}e_{\rm{av}}^{2}[k]\nonumber \\
    &-\left(1 -\frac{\epsilon a^{2} H^{\ast}K}{2}\right)\epsilon a^{2} H^{\ast}K\hat{G}_{\rm{av}}[k]e_{\rm{av}}[k]-\hat{G}_{\rm{av}}^{2}[k] \,.\label{eq:lyapfunc_1}
\end{align}
Then, an upper bound for (\ref{eq:lyapfunc_1}) can be found as
\begin{align}
    \Delta V[\hat{G}_{\rm{av}}[k]]&\leq -\left(1-\left(1 -\frac{\epsilon a^{2} H^{\ast}K}{2}\right)^{2}\right)\hat{G}_{\rm{av}}^{2}[k] \nonumber \\
    &+\left(1 -\frac{\epsilon a^{2} H^{\ast}K}{2}\right)\epsilon a^{2} |H^{\ast}||K||\hat{G}_{\rm{av}}[k]||e_{\rm{av}}[k]|\nonumber \\
    &+\frac{\epsilon^{2} a^{4} H^{\ast 2}K^{2}}{4}e_{\rm{av}}^{2}[k]\,. \label{eq:lyapfunc_2}
\end{align}
Now, applying the Peter-Paul inequality \cite{W:1971}, $cd\leq \frac{\varepsilon c^2}{2} +\frac{d^2}{2\varepsilon}$, for all $c,d,\varepsilon>0$, with $c=|\hat{G}^{\top}_{\rm{av}}[k]|$, $$d=\left(1 -\frac{\epsilon a^{2} H^{\ast}K}{2}\right)a^{2} |H^{\ast}||K||e_{\rm{av}}[k]|$$ and $$\varepsilon=\frac{\left(1-\left(1 -\frac{\epsilon a^{2} H^{\ast}K}{2}\right)^{2}\right)}{2\epsilon},$$ inequality (\ref{eq:lyapfunc_2}) is upper bounded by
 \begin{align}
    &\Delta V[\hat{G}_{\rm{av}}[k]]  \leq -\frac{1}{2}\left(1-\left(1 -\frac{\epsilon a^{2} H^{\ast}K}{2}\right)^{2}\right)\hat{G}_{\rm{av}}^{2}[k] \nonumber \\
    &+\frac{2\left(1 -\frac{\epsilon a^{2} H^{\ast}K}{2}\right)^{2}\epsilon^{2} a^{4} |H^{\ast}|^{2}|K|^{2}}{\left(1-\left(1 -\frac{\epsilon a^{2} H^{\ast}K}{2}\right)^{2}\right)}|e_{\rm{av}}[k]|^{2}\nonumber \\
    &+\frac{\epsilon^{2} a^{4} H^{\ast 2}K^{2}}{4}e_{\rm{av}}^{2}[k] \nonumber \\
    &=-\frac{1}{2}\left(1-\left(1 -\frac{\epsilon a^{2} H^{\ast}K}{2}\right)^{2}\right)\left(\hat{G}_{\rm{av}}^{2}[k]\right. \nonumber \\
    &\left.-\frac{\epsilon^{2} a^{4} H^{\ast 2}K^{2}}{2}\frac{\left(1+7\left(1 -\frac{\epsilon a^{2} H^{\ast}K}{2}\right)^{2}\right)}{\left(1-\left(1 -\frac{\epsilon a^{2} H^{\ast}K}{2}\right)^{2}\right)^{2}}e_{\rm{av}}^{2}[k]\right) \,. \label{eq:lyapfunc_3}
\end{align}   
By using (\ref{eq:alpha}), inequality (\ref{eq:lyapfunc_3}) is upper bounded by
 \begin{align}
    &\Delta V[\hat{G}_{\rm{av}}[k]] \leq   \nonumber \\
    &-\frac{1}{2}\left(1-\left(1 -\frac{\epsilon a^{2} H^{\ast}K}{2}\right)^{2}\right)\left(|\hat{G}_{\rm{av}}[k]|^2 - \alpha^{2} |e_{\rm{av}}[k]|^{2}\right) \nonumber \\
    &=-\frac{1}{2}\left(1-\left(1 -\frac{\epsilon a^{2} H^{\ast}K}{2}\right)^{2}\right)\left(|\hat{G}_{\rm{av}}[k]| + \alpha |e_{\rm{av}}[k]|\right) \nonumber \\
    &\quad\times \left(|\hat{G}_{\rm{av}}[k]| - \alpha |e_{\rm{av}}[k]|\right)\,. \label{eq:lyapfunc_4}
\end{align}
In the proposed event-triggering mechanism, the update law is (\ref{eq:U_MD2}) such that the vector $u_{\rm{av}}[k]$ is held constant between two consecutive events, and therefore
\begin{align}
\alpha|e_{\rm{av}}[k]|& \leq \sqrt{\sigma}|\hat{G}_{\rm{av}}[k]|\,, \quad \sigma \in (0,1)\,. \label{eq:eAv_upperBound}
\end{align}
Now, by using (\ref{eq:eAv_upperBound}), inequality (\ref{eq:lyapfunc_4}) is upper bounded as 
 \begin{align}
    &\Delta V[\hat{G}_{\rm{av}}[k]] \leq  \nonumber \\
    &-\frac{1}{2}\left(1-\left(1 -\frac{\epsilon a^{2} H^{\ast}K}{2}\right)^{2}\right)\left(|\hat{G}_{\rm{av}}[k]| + \sqrt{\sigma}|\hat{G}_{\rm{av}}[k]|\right) \nonumber \\
    &\quad\times \left(|\hat{G}_{\rm{av}}[k]| - \sqrt{\sigma}|\hat{G}_{\rm{av}}[k]|\right) \nonumber \\
    &=-\frac{1}{2}\left(1-\left(1 -\frac{\epsilon a^{2} H^{\ast}K}{2}\right)^{2}\right)\left(1 - \sigma\right)|\hat{G}_{\rm{av}}[k]|^2  \nonumber \\
    &=-\frac{1}{2}\left(1-\left(1 -\frac{\epsilon a^{2} H^{\ast}K}{2}\right)^{2}\right)\left(1 - \sigma\right)V[\hat{G}_{\rm{av}}[k]]\,. \label{eq:lyapfunc_5}
\end{align}
From (\ref{eq:lyapfunc_0}), $V[\hat{G}_{\rm{av}}[k+1]]= V[\hat{G}_{\rm{av}}[k]]+\Delta V[\hat{G}_{\rm{av}}[k]]$, and by using (\ref{eq:lyapfunc_5}), one has 
 \begin{align}
 &V[\hat{G}_{\rm{av}}[k+1]]= V[\hat{G}_{\rm{av}}[k]]+\Delta V[\hat{G}_{\rm{av}}[k]] \nonumber \\
    &\leq \left(1- \left(1-\left(1 -\frac{\epsilon a^{2} H^{\ast}K}{2}\right)^{2}\right)\frac{\left(1 - \sigma\right)}{2}\right)V[\hat{G}_{\rm{av}}[k]]\,. \label{eq:lyapfunc_6}
\end{align}
The recursive structure of (\ref{eq:lyapfunc_6}) allows us to derive an upper bound for the evolution of the Lyapunov function (\ref{eq:lyapfunc}) over successive intervals such that we can iteratively relate the Lyapunov function at the current iteration step $k$ to its value at earlier triggering times. This approach lead us to
 \begin{align}
 V[\hat{G}_{\rm{av}}[k]]&\mathbb{\leq} \left(1\mathbb{-} \left(1\mathbb{-}\left(1 \mathbb{-}\frac{\epsilon a^{2} H^{\ast}K}{2}\right)^{2}\right)\frac{\left(1 \mathbb{-} \sigma\right)}{2}\right)^{k} \nonumber \\
 &\quad \times V[\hat{G}_{\rm{av}}[0]]\,.\label{eq:lyapfunc_7}
\end{align}
Therefore, by using (\ref{eq:lyapfunc}), we obtain 
\begin{align}
 |\hat{G}_{\rm{av}}[k]|\mathbb{\leq} \left(1\mathbb{-} \left(1\mathbb{-}\left(1 \mathbb{-}\frac{\epsilon a^{2} H^{\ast}K}{2}\right)^{2}\right)\frac{\left(1 \mathbb{-} \sigma\right)}{2}\right)^{\frac{k}{2}} \!\!|\hat{G}_{\rm{av}}[0]|\,.\label{eq:lyapfunc_8} 
\end{align}
Moreover, from (\ref{eq:hatG_av_k_1}), $|\hat{G}_{\rm{av}}[k]|=|H^{\ast}||\tilde{\theta}_{\rm{av}}[k]|$, and, therefore, from (\ref{eq:lyapfunc_8}), we can state 
\begin{align}
 |\tilde{\theta}_{\rm{av}}[k]|\mathbb{\leq} \left(1\mathbb{-} \left(1\mathbb{-}\left(1 \mathbb{-}\frac{\epsilon a^{2} H^{\ast}K}{2}\right)^{2}\right)\frac{\left(1 \mathbb{-} \sigma\right)}{2}\right)^{\frac{k}{2}} \!\!|\tilde{\theta}_{\rm{av}}[0]|\,.\label{eq:lyapfunc_9} 
\end{align}

Since (\ref{eq:f2_20260321_1}) is characterized by a $T$-periodic discrete-time vector field and noting that the corresponding averaged system with state variable $\tilde{\theta}_{\rm{av}}[k]$ is exponentially stable as established in (\ref{eq:lyapfunc_9}), the conditions of \cite[Theorem~2.2.1]{BFS:1988} 
are satisfied. In particular, for sufficiently small $\epsilon > 0$ and initial conditions $\tilde{\theta}[0]$ sufficiently close to the origin, the solutions of the original system and the averaged system remain close over finite time intervals. Therefore,
\begin{align}
|\tilde{\theta}[k]-\tilde{\theta}_{\rm{av}}[k]|\leq\mathcal{O}\left(\epsilon\right)\,. \label{eq:plotnikov}
\end{align}
Now, adding and subtracting $\tilde{\theta}_{\rm{av}}[k]$ in the right-hand side of (\ref{eq:theta_2_event_siso}), one has
\begin{align}
\theta[k] - \theta^\ast = \tilde{\theta}_{\rm{av}}[k]+\tilde{\theta}[k]-\tilde{\theta}_{\rm{av}}[k] + a\sin[\omega \epsilon k]\,, \label{eq:q_ETSSC_v3cu}
\end{align}  
whose norm can be upper bounded by using the triangle inequality \cite{A:1957}, leading to 
\begin{align}
|\theta[k] - \theta^\ast| \leq |\tilde{\theta}_{\rm{av}}[k]|+|\tilde{\theta}[k]-\tilde{\theta}_{\rm{av}}[k]| + |a\sin[\omega \epsilon k]|\,.\label{eq:q_ETSSC_v3cu2}
\end{align} 
Thus, by using (\ref{eq:lyapfunc_9}) and (\ref{eq:plotnikov}), inequality (\ref{eq:q_ETSSC_v3cu2}) is upper bounded by 
\begin{align}
|\theta[k]\mathbb{-}\theta^{\ast}|  &\mathbb{\leq}  \left(1\mathbb{-} \left(1\mathbb{-}\left(1 \mathbb{-}\frac{\epsilon a^{2} H^{\ast}K}{2}\right)^{2}\right)\frac{\left(1 \mathbb{-} \sigma\right)}{2}\right)^{\frac{k}{2}} \nonumber \\
 &\quad\times|\theta[0]-\theta^{\ast}|+\mathcal{O}\left(a+\epsilon\right)\,,\label{eq:q_ETSSC_v4}
\end{align} 
then, inequality (\ref{eq:normTheta_thm1}) is verified.

Now, since $y[k]=Q[\theta[k]]$, from (\ref{eq:Q_1_event_siso}) and \ref{eq:q_ETSSC_v4}), one has 
\begin{align}
    |y[k] \mathbb{-} Q^{\ast}|&\mathbb{=}\frac{1}{2}|H^{\ast}|(\theta[k]-\theta^{\ast})^{2} \label{eq:normy_siso}\\
    &\mathbb{\leq} \frac{1}{2}|H^{\ast}|\left(1\mathbb{-} \left(1\mathbb{-}\left(1 \mathbb{-}\frac{\epsilon a^{2} H^{\ast}K}{2}\right)^{2}\right)\frac{\left(1 \mathbb{-} \sigma\right)}{2}\right)^{k} \nonumber \\
 &\quad\times|\theta[0]-\theta^{\ast}|^{2}+\mathcal{O}\left(a+\epsilon\right)^{2} \nonumber \\
 &\quad \mathbb{+}|H^{\ast}|\left(1\mathbb{-} \left(1\mathbb{-}\left(1 \mathbb{-}\frac{\epsilon a^{2} H^{\ast}K}{2}\right)^{2}\right)\frac{\left(1 \mathbb{-} \sigma\right)}{2}\right)^{\frac{k}{2}} \nonumber \\
 &\quad\times|\theta[0]-\theta^{\ast}|\mathcal{O}\left(a+\epsilon\right)\,. \label{eq:yQ_1}
\end{align}
Applying again the Peter-Paul inequality \cite{W:1971}, $cd\leq \frac{\varepsilon c^2}{2} +\frac{d^2}{2\varepsilon}$ for all $c,d,\varepsilon>0$, with \mbox{$c=\left(1- \left(1\mathbb{-}\left(1 -\frac{\epsilon a^{2} H^{\ast}K}{2}\right)^{2}\right)\frac{\left(1 - \sigma\right)}{2}\right)^{\frac{k}{2}}|\theta[0]-\theta^{\ast}|$}, $d=\mathcal{O}\left(a+\epsilon\right)$ and $\varepsilon=\frac{1}{2}$, inequality (\ref{eq:yQ_1}) is upper bounded by
\begin{align}
    |y[k] - Q^{\ast}|& \leq 2\left(1\mathbb{-} \left(1\mathbb{-}\left(1 \mathbb{-}\frac{\epsilon a^{2} H^{\ast}K}{2}\right)^{2}\right)\frac{\left(1 \mathbb{-} \sigma\right)}{2}\right)^{k} \nonumber \\
 &\quad\mathbb{\times}\frac{1}{2}|H^{\ast}||\theta[0]-\theta^{\ast}|^{2}+(1+|H^{\ast}|)\mathcal{O}\left(a+\epsilon\right)^{2} \! . \label{eq:yQ_2}
\end{align}
From (\ref{eq:normy_siso}), $\frac{1}{2}|H^{\ast}||\theta[0]-\theta^{\ast}|^{2}=|y[0]-Q^{\ast}|$ and, according to \cite[Definition~10.1]{K:2002}, $(1+|H^{\ast}|)\mathcal{O}\left(a+\epsilon\right)^{2}$ remains as an order of magnitude $\mathcal{O}\left(a^{2}+\epsilon^2\right)$. Moreover, $\mathcal{O}\left(a+\epsilon\right)^{2}\leq \mathcal{O}\left(a^{2}+\epsilon^{2}\right)$. Thus, inequality (\ref{eq:yQ_2}) is upper bounded by
\begin{align}
    |y[k] - Q^{\ast}|&\leq 2\left(1\mathbb{-} \left(1\mathbb{-}\left(1 \mathbb{-}\frac{\epsilon a^{2} H^{\ast}K}{2}\right)^{2}\right)\frac{\left(1 \mathbb{-} \sigma\right)}{2}\right)^{k} \nonumber \\
 &\quad\mathbb{\times}|y[0]-Q^{\ast}|+\mathcal{O}\left(a^2+\epsilon^2\right)\,,  \label{eq:yQ_3}
\end{align}
such that inequality (\ref{eq:normY_thm1}) is also verified.

\begin{flushleft}
\underline{\it B. Avoidance of Zeno-Like Behavior}
\end{flushleft}

We now investigate the existence of a minimum interval between two consecutive triggering instants in the considered discrete-time framework. Due to the discrete-time nature of the closed-loop system, inter-execution times are inherently lower bounded by one sampling step, and thus Zeno behavior is excluded by construction.

Accordingly, the objective here is not only to establish Zeno avoidance, but rather to ensure that the triggering mechanism yields a well-posed and nontrivial execution pattern. In particular, it is required that the event-triggering condition does not lead to pathological behaviors such as persistent triggering at every iteration, which would eliminate any computational or communication advantages.

Now, from (\ref{eq:hatG_av_k+1_1}) and (\ref{eq:e_av_k_1}), we can write, for all $k \in [k_l, k_{l+1})$:  
\begin{align}
\hat{G}_{\rm{av}}[k+1]&=\left(1 \mathbb{-}\frac{\epsilon a^{2} H^{\ast}K}{2}\right)\hat{G}_{\rm{av}}[k]-\frac{\epsilon a^{2} H^{\ast}K}{2}e_{\rm{av}}[k]\,, \label{eq:hatG_av_k+1_2} \\
&=\hat{G}_{\rm{av}}[k]-\frac{\epsilon a^{2} H^{\ast}K}{2}\hat{G}_{\rm{av}}[k_l] \,, \label{eq:hatG_av_k+1_3} \\
&=\hat{G}_{\rm{av}}[k_{l}]-e_{\rm{av}}[k]-\frac{\epsilon a^{2} H^{\ast}K}{2}\hat{G}_{\rm{av}}[k_l] \,, \label{eq:hatG_av_k+1_3_20260330} \\
e_{\rm{av}}[k+1] &= \hat{G}_{\rm{av}}[k_l] - \hat{G}_{\rm{av}}[k+1]\,. 
\label{eq:e_av_k_2}
\end{align}
By using (\ref{eq:hatG_av_k+1_3_20260330}) and (\ref{eq:e_av_k_1}), one can conclude $\hat{G}_{\rm{av}}[k]=\hat{G}_{\rm{av}}[k_{l}]-e_{\rm{av}}[k]$ such that (\ref{eq:e_av_k_2}) can be rewritten as
\begin{align}
e_{\rm{av}}[k+1] &= \hat{G}_{\rm{av}}[k_l] -\hat{G}_{\rm{av}}[k+1] \nonumber \\
&= e_{\rm{av}}[k]+\frac{\epsilon a^{2} H^{\ast}K}{2}\hat{G}_{\rm{av}}[k_l] \,. 
\label{eq:e_av_k_3}
\end{align}
Now, the iteration of (\ref{eq:hatG_av_k+1_3}) and (\ref{eq:e_av_k_3}) for $k=\{k_{l},k_{l}+1,k_{l}+2,\ldots\}$, with $k \in [k_{l},k_{l+1})$, together with the condition $e[k_{l}]=0$, yields
\begin{align}
\hat{G}_{\rm{av}}[k_{l}+1]&=\hat{G}_{\rm{av}}[k_{l}]-\frac{\epsilon a^{2} H^{\ast}K}{2}\hat{G}_{\rm{av}}[k_l] \nonumber \\
&=\left(1-\frac{\epsilon a^{2} H^{\ast}K}{2}\right)\hat{G}_{\rm{av}}[k_l] \,, \nonumber \\ 
\hat{G}_{\rm{av}}[k_{l}+2]&=\hat{G}_{\rm{av}}[k_{l}+1]-\frac{\epsilon a^{2} H^{\ast}K}{2}\hat{G}_{\rm{av}}[k_l] \nonumber \\ 
&=\left(1-2\frac{\epsilon a^{2} H^{\ast}K}{2}\right)\hat{G}_{\rm{av}}[k_l]\,, \nonumber \\ 
\hat{G}_{\rm{av}}[k_{l}+3]&=\hat{G}_{\rm{av}}[k_{l}+2]-\frac{\epsilon a^{2} H^{\ast}K}{2}\hat{G}_{\rm{av}}[k_l] \nonumber \\
&=\left(1-3\frac{\epsilon a^{2} H^{\ast}K}{2}\right)\hat{G}_{\rm{av}}[k_l] \,, \nonumber \\
\vdots ~~~~ & ~~~~~~~~~~~~~~~~~~~~~~~~~~ \vdots \nonumber \\
\hat{G}_{\rm{av}}[k]&=\left(1-(k-k_{l})\frac{\epsilon a^{2} H^{\ast}K}{2}\right)\hat{G}_{\rm{av}}[k_l]\,, \label{eq:hatG_av_k_2}
\end{align}
and
\begin{align}
e_{\rm{av}}[k_{l}+1] &= e_{\rm{av}}[k_{l}]+\frac{\epsilon a^{2} H^{\ast}K}{2}\hat{G}_{\rm{av}}[k_l] \nonumber \\
&= \frac{\epsilon a^{2} H^{\ast}K}{2}\hat{G}_{\rm{av}}[k_l]\,, \nonumber \\
e_{\rm{av}}[k_{l}+2] &=   e_{\rm{av}}[k_{l}+1]+\frac{\epsilon a^{2} H^{\ast}K}{2}\hat{G}_{\rm{av}}[k_l] \nonumber \\
&=2\frac{\epsilon a^{2} H^{\ast}K}{2}\hat{G}_{\rm{av}}[k_l]\,,\nonumber \\
e_{\rm{av}}[k_{l}+3] &= e_{\rm{av}}[k_{l}+2]+\frac{\epsilon a^{2} H^{\ast}K}{2}\hat{G}_{\rm{av}}[k_l]\nonumber \\
&= 3\frac{\epsilon a^{2} H^{\ast}K}{2}\hat{G}_{\rm{av}}[k_l] \nonumber \\
\vdots ~~~~ & ~~~~~~~~~~~~~~~~~~~~~~~~~~ \vdots \nonumber \\
e_{\rm{av}}[k] &= (k-k_{l})\frac{\epsilon a^{2} H^{\ast}K}{2}\hat{G}_{\rm{av}}[k_l]\,.
\label{eq:e_av_k_4}
\end{align}
Now, the idea is to combine an upper bound for $|e[k]|$ with a lower bound of $|\hat{G}[k]|$ in order to determine when the event-triggering condition is violated. Therefore, by invoking the average method for discrete-time systems \cite[Theorem~2.2.1]{BFS:1988}, we have 
\begin{align}
|\hat{G}[k]-\hat{G}_{\rm{av}}[k]|&\leq\mathcal{O}\left(\epsilon\right)\,, \label{eq:sastry_3} \\
|e[k]-e_{\rm{av}}[k]|&\leq\mathcal{O}\left(\epsilon\right)\,, \label{eq:sastry_4}
\end{align}
such that, by using the triangle inequality \cite{A:1957}, one additionally writes 
\begin{align}
|\hat{G}[k]|&=|\hat{G}_{\rm{av}}[k]+\hat{G}[k]-\hat{G}_{\rm{av}}[k]| \nonumber \\
&\geq \left||\hat{G}_{\rm{av}}[k]|-|\hat{G}[k]-\hat{G}_{\rm{av}}[k]|\right| \,, \label{eq:normG_k_1} \\
|e[k]|&=|e_{\rm{av}}[k]+e[k]-e_{\rm{av}}[k]| \nonumber \\
&\leq |e_{\rm{av}}[k]|+|e[k]-e_{\rm{av}}[k]| \,. \label{eq:normE_k_1}
\end{align}  
Using (\ref{eq:hatG_av_k_2}), (\ref{eq:e_av_k_4}), (\ref{eq:sastry_3}) and (\ref{eq:sastry_4}), the following bounds are obtained for (\ref{eq:normG_k_1}) and (\ref{eq:normE_k_1}), respectively:
\begin{align}
|\hat{G}[k]|&\geq \left|\left|\left(1-(k-k_{l})\frac{\epsilon a^{2} H^{\ast}K}{2}\right)\hat{G}_{\rm{av}}[k_l]\right|-\mathcal{O}\left(\epsilon\right)\right| \nonumber \\
&=\underline{\hat{G}}[k]\,, \label{eq:normG_k_2} \\
|e[k]|&\leq  \left|(k-k_{l})\frac{\epsilon a^{2} H^{\ast}K}{2}\hat{G}_{\rm{av}}[k_l]\right|+\mathcal{O}\left(\epsilon\right) \nonumber \\
&=\bar{e}[k]\,. \label{eq:normE_k_2}
\end{align}  
At this stage, we combine the upper bound for the error $\bar{e}[k]$ in (\ref{eq:normE_k_2}) with the lower bound of the gradient estimate $\underline{\hat{G}}[k]$ in order to determine the minimum number of iterations after which the event-triggering condition is necessarily violated. In other words, if at some iteration $k$ one has $\sqrt{\sigma} \underline{\hat{G}}[k]-\alpha\bar{e}[k] \leq 0 $, then the condition (\ref{eq:k_next_event}) can no longer be ensured, and thus an event must occur. Therefore, the first iteration $k^{\ast}$ such that $\sqrt{\sigma} \underline{\hat{G}}[k]-\alpha\bar{e}[k] \leq 0 $ provides an estimate of the minimum number of iterations after which the event-triggering condition is necessarily violated, {\it i.e.},
\begin{align}
    k^{\ast}&= \arg \min_{k\in\mathbb{N}}\left\{\alpha\left(\left|(k-k_{l})\frac{\epsilon a^{2} H^{\ast}K}{2}\hat{G}_{\rm{av}}[k_l]\right|+\mathcal{O}\left(\epsilon\right)\right)\geq \right.\nonumber \\
    &\quad\left.\sqrt{\sigma}\left|\left|\left(1-(k-k_{l})\frac{\epsilon a^{2} H^{\ast}K}{2}\right)\hat{G}_{\rm{av}}[k_l]\right|-\mathcal{O}\left(\epsilon\right)\right|\right\}\,.
\end{align}
%
Hence, the proof is complete. 
\end{proof} 


\section{Simulation results} \label{sec:sim_siso}

To illustrate the main features of the proposed discrete-time event-triggered extremum-seeking strategy, consider the nonlinear map in (\ref{eq:Q_1_event_siso}), with input $\theta[k]\in \mathbb{R}$, output $y[k]\in \mathbb{R}$, and unknown parameters $H^{\ast}=-0.7$, $Q^{\ast}=2$, and $\theta^{\ast}=3$. The dither signal is selected with parameters $a=0.1$ and $\omega=7$ [rad/sec], while the event-triggering parameters are chosen as $\sigma=0.7$ and $\alpha=0.74$. The control gain is set to $K=-240$, the step size is $\epsilon=0.18$ [sec], and the initial condition is $\hat{\theta}[0]=0.5$.

In Figures~\ref{fig:theta_siso} and \ref{fig:y_siso}, we can check $\theta[k]$ converging to $\theta^{\ast}$ and the output $y[k]$ approaching its corresponding extremum value. 

Figure~\ref{fig:U_siso} shows the aperiodic pattern of the control update instants, characterized by the nonuniform distribution of inter-execution intervals inherent to the event-triggered implementation. In this framework, the control input is updated only when the triggering condition is satisfied. Over 1000 iterations, the control law is updated only 19 times, yielding an average inter-execution interval of $9.47$ seconds.

Figure~\ref{fig:hatG_siso} presents the evolution of the gradient estimate, showing that it converges to zero while remaining well behaved along the iterations. The event-triggered mechanism ensures that all signals remain bounded and exhibit smooth convergence, as expected.

\newpage
These results highlight the effectiveness of the proposed discrete-time event-triggered extremum-seeking scheme in significantly reducing the number of control updates while still ensuring convergence to the extremum.

\begin{figure}[h!]
	\centering
	\subfigure[Input of the nonlinear map, \mbox{$\theta[k]$}. \label{fig:theta_siso}]{\includegraphics[width=5.5cm]{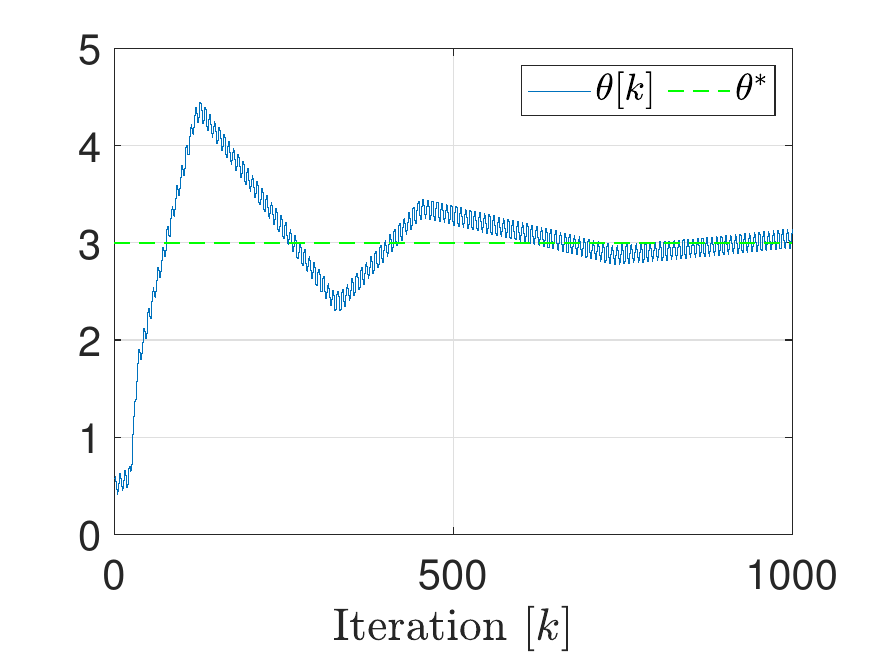}}
    \\
	\subfigure[Output of the nonlinear map, \mbox{$y[k]$}. \label{fig:y_siso}]{\includegraphics[width=5.5cm]{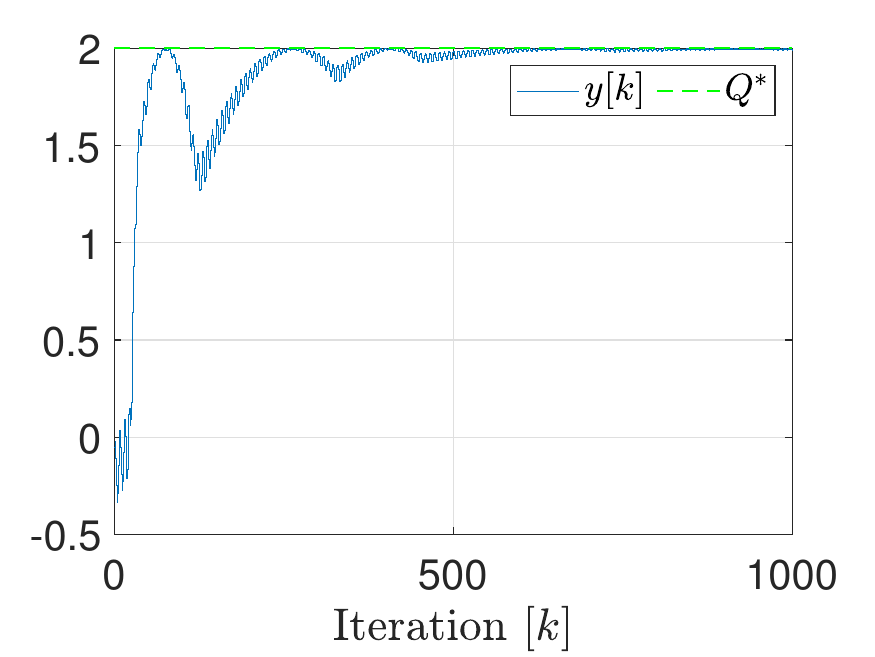}}
	\\
	\subfigure[Control input, \mbox{$u[k]$} \label{fig:U_siso}]{\includegraphics[width=5.5cm]{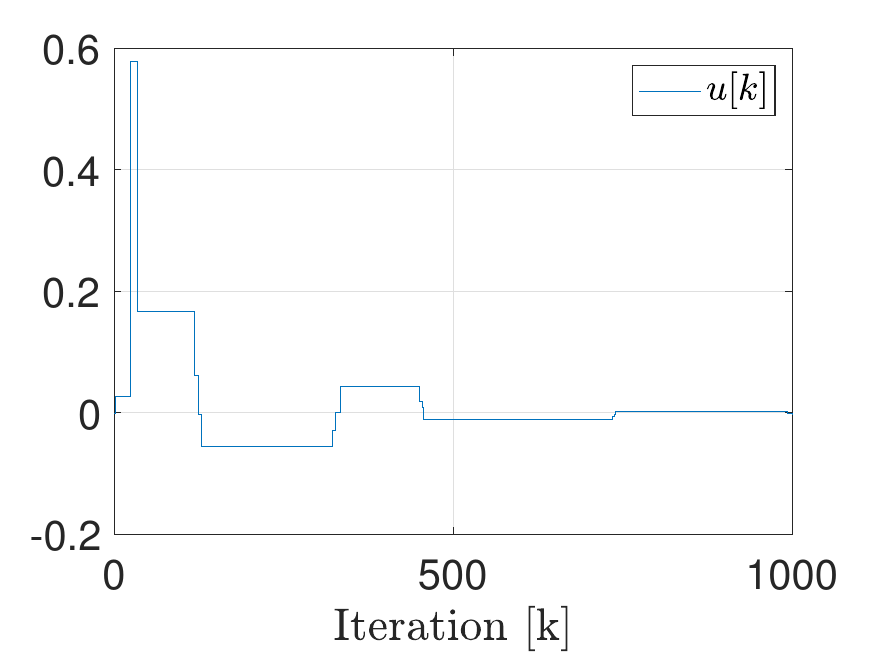}}
	\\
	\subfigure[Gradient estimate, \mbox{$\hat{G}[k]$}. \label{fig:hatG_siso}]{\includegraphics[width=5.5cm]{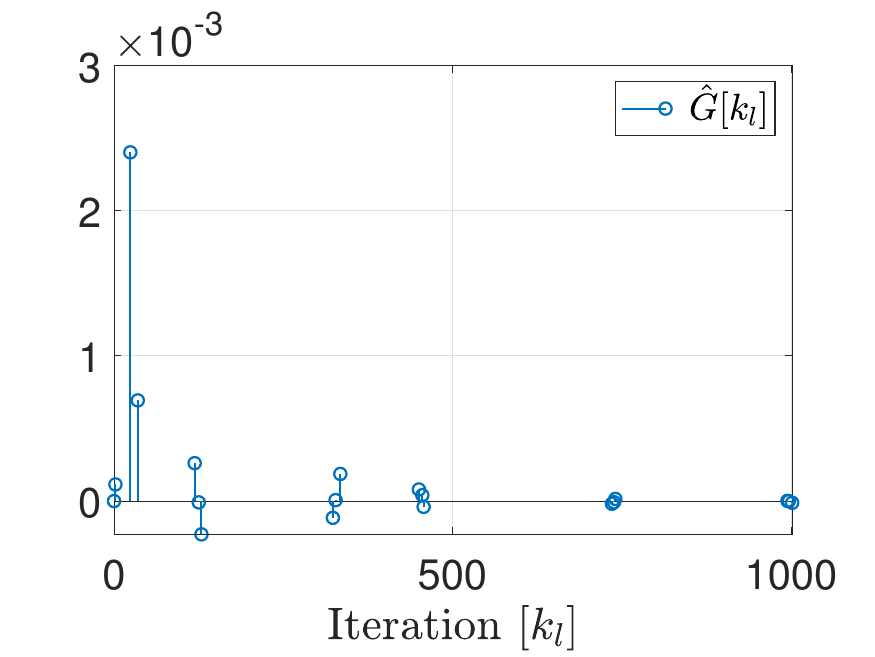}}   
	\caption{Simulations of the Discrete-time Event-triggered Extremum Seeking Control System. \label{fig:NESC_siso}}
    \begin{picture}(300,0)(-120,-490)
    \includegraphics[width=2.2cm]{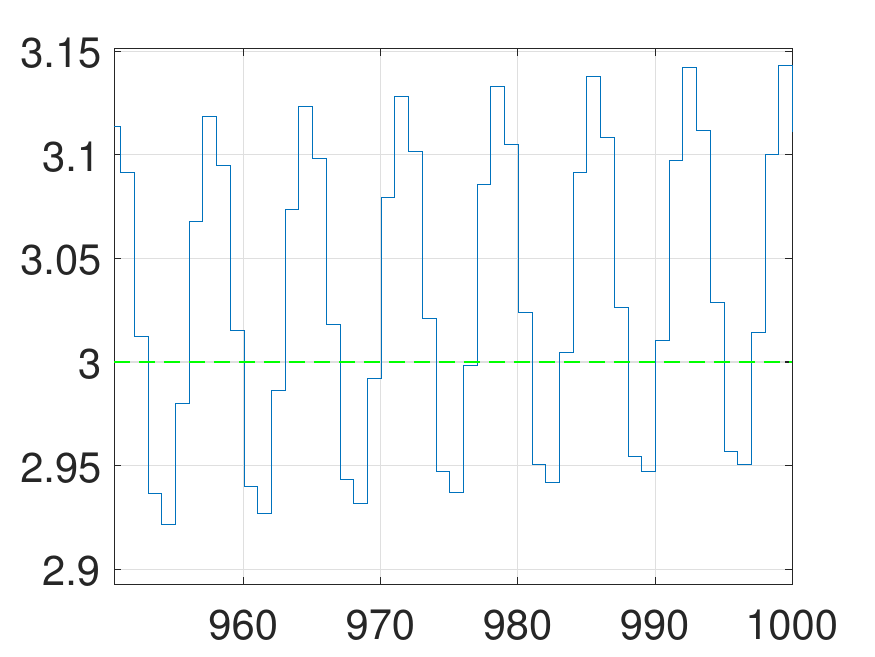}
    \end{picture}
    \begin{picture}(300,0)(-90,-362.5)
    \includegraphics[width=3.2cm]{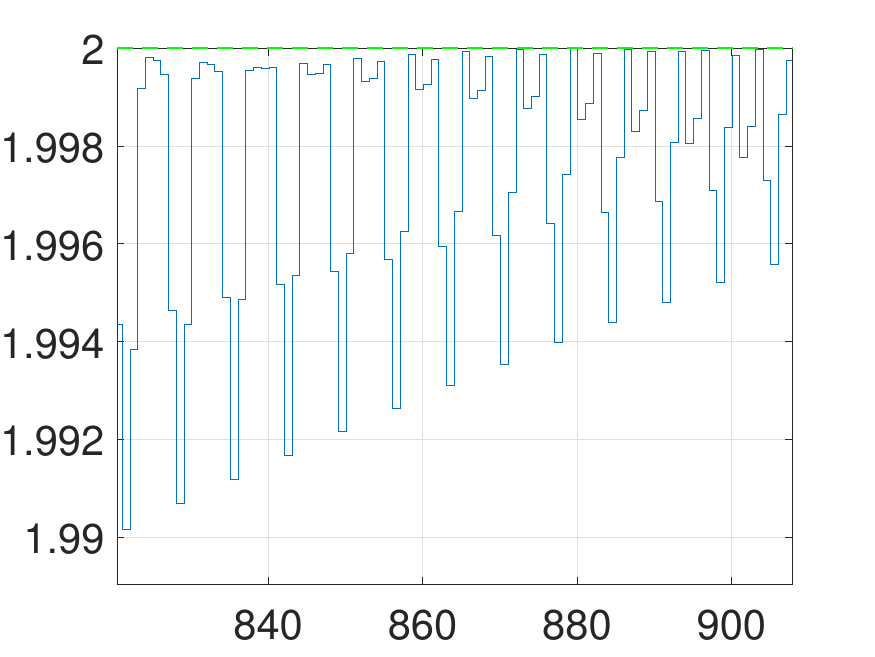}
    \end{picture}
    \begin{picture}(300,0)(-83,-258)
    \includegraphics[width=2.7cm]{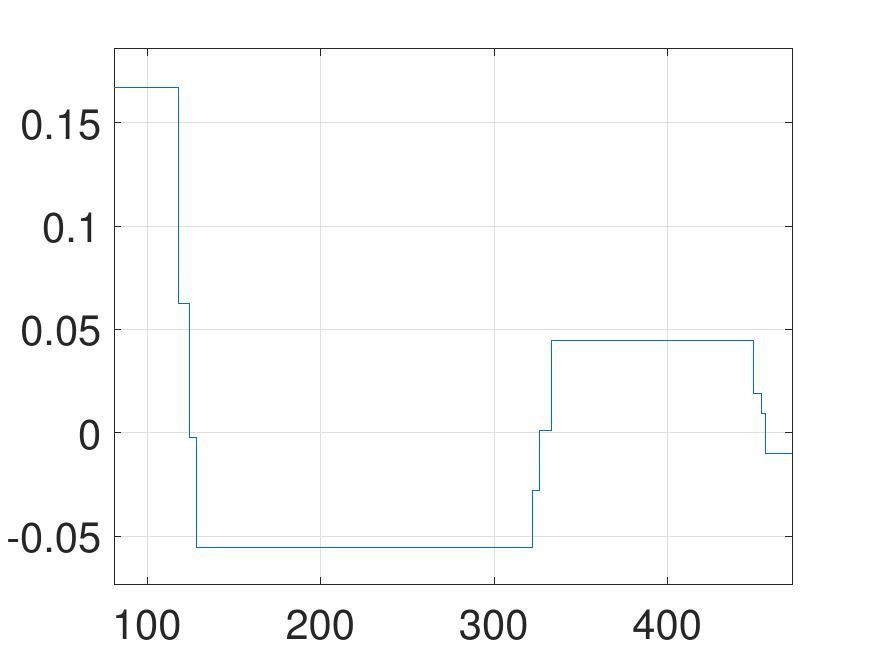}
    \end{picture}
    \begin{picture}(300,0)(-75,-116)
    \includegraphics[width=2.9cm]{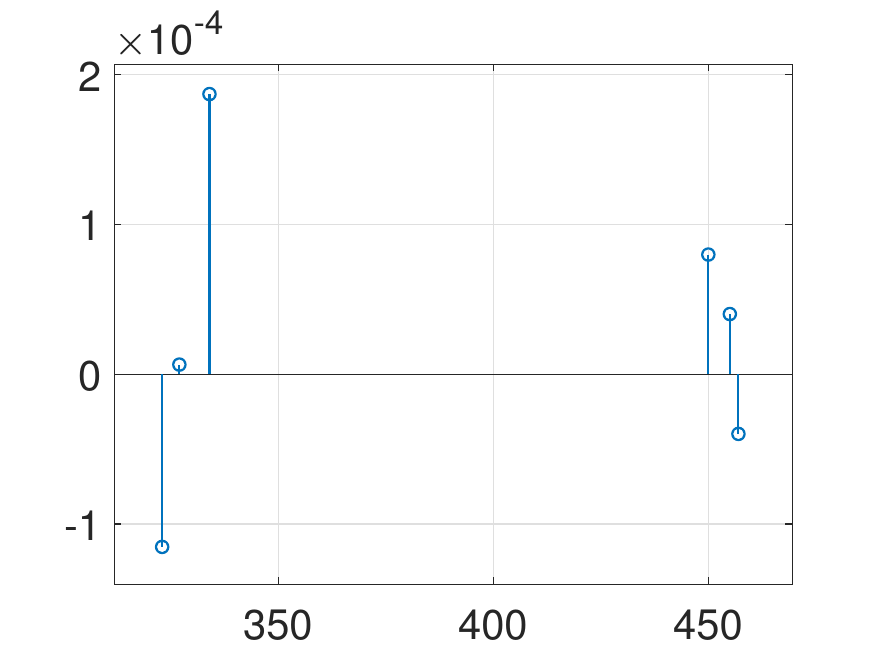}
    \end{picture}
\end{figure}

 \newpage

\section{Conclusions}

This paper proposed a discrete-time event-triggered extremum seeking scheme for real-time optimization of static nonlinear maps, where control updates occur only when required by a state-dependent triggering condition. By combining discrete-time averaging and Lyapunov analysis, we established practical convergence to a neighborhood of the unknown extremum and exponential stability of the average dynamics, showing that the triggering mechanism preserves the essential optimization properties of classical extremum seeking while significantly reducing the number of input updates. Simulation results confirmed that the proposed approach achieves comparable optimization performance with substantially fewer actuation events, highlighting its potential for resource-aware digital implementations.

Future work may explore extensions to multi-agent and networked extremum seeking settings, robustness with respect to measurement noise and delays, and experimental validation in embedded digital platforms, where the advantages of discrete-time event-triggered implementations are expected to be even more pronounced.




\begin{thebibliography}{10}
\providecommand{\url}[1]{#1}
\csname url@samestyle\endcsname
\providecommand{\newblock}{\relax}
\providecommand{\bibinfo}[2]{#2}
\providecommand{\BIBentrySTDinterwordspacing}{\spaceskip=0pt\relax}
\providecommand{\BIBentryALTinterwordstretchfactor}{4}
\providecommand{\BIBentryALTinterwordspacing}{\spaceskip=\fontdimen2\font plus
\BIBentryALTinterwordstretchfactor\fontdimen3\font minus
  \fontdimen4\font\relax}
\providecommand{\BIBforeignlanguage}[2]{{%
\expandafter\ifx\csname l@#1\endcsname\relax
\typeout{** WARNING: IEEEtran.bst: No hyphenation pattern has been}%
\typeout{** loaded for the language `#1'. Using the pattern for}%
\typeout{** the default language instead.}%
\else
\language=\csname l@#1\endcsname
\fi
#2}}
\providecommand{\BIBdecl}{\relax}
\BIBdecl

\bibitem{KW:2000}
M.~Krstic and H.-H. Wang, ``Stability of extremum seeking feedback for general
  nonlinear dynamic systems,'' \emph{Automatica}, vol.~36, pp. 595--601, 2000.

\bibitem{AK:2003}
K.~B. Ariyur and M.~Krstic, \emph{Real-Time Optimization by Extremum-Seeking
  Control}.\hskip 1em plus 0.5em minus 0.4em\relax Wiley, 2003.

\bibitem{TRoux:2022}
T.~R. Oliveira and M.~Krstic, \emph{Extremum Seeking through Delays and
  PDEs}.\hskip 1em plus 0.5em minus 0.4em\relax Society for Industrial and
  Applied Mathematics, 2022.

\bibitem{T:2007}
P.~Tabuada, ``Event-triggered real-time scheduling of stabilizing control
  tasks,'' \emph{IEEE Transactions on Automatic Control}, vol.~52, pp.
  1680--1685, 2007.

\bibitem{VHPR:2023a}
V.~H.~P. Rodrigues, T.~R. Oliveira, L.~Hsu, M.~Diagne, and M.~Krstic,
  ``Event-triggered and periodic event-triggered extremum seeking control,''
  \emph{Automatica}, vol. 174, p. 112161, 2025.

\bibitem{BFS:1988}
E.-W. Bai, L.-C. Fu, and S.~Sastry, ``Averaging analysis for discrete time and
  sampled data adaptive systems,'' \emph{IEEE Transactions on Circuits and
  Systems}, vol.~35, no.~2, pp. 137--148, 1988.

\bibitem{PPY:2004}
V.~A. Plotnikov, L.~I. Plotnikova, and A.~T. Yarovoi, ``Averaging method for
  discrete systems and its application to control problems,'' \emph{Nonlinear
  Oscillations}, vol.~7, no.~2, pp. 240--253, 2004.

\bibitem{EDK:2010}
A.~Eqtami, D.~V. Dimarogonas, and K.~J. Kyriakopoulos, ``Event-triggered
  control for discrete-time systems,'' in \emph{Proceedings of the 2010
  American Control Conference}, 2010, pp. 4719--4724.

\bibitem{frank:2023}
X.~Wang, J.~Berberich, J.~Sun, G.~Wang, F.~Allg{\" o}wer, and J.~Chen,
  ``Model-based and data-driven control of event-and self-triggered
  discrete-time linear systems,'' \emph{IEEE Transactions on Cybernetics},
  vol.~53, no.~9, pp. 6066--6079, 2023.

\bibitem{frank:2023b}
S.~Wildhagen, J.~Berberich, M.~Hertneck, and F.~Allg{\" o}wer, ``Data-driven
  analysis and controller design for discrete-time systems under aperiodic
  sampling,'' \emph{IEEE Transactions on Automatic Control}, vol.~68, no.~6,
  pp. 3210--3225, 2023.

\bibitem{AW:1997}
K.~J. {\r A}str{\" o}m and B.~Wittenmark, \emph{Computer-Controlled Systems:
  Theory and Design}.\hskip 1em plus 0.5em minus 0.4em\relax Prentice Hall,
  1997.

\bibitem{GKN:2012}
A.~Ghaffari, M.~Krstic, and D.~Nesic, ``Multivariable {Newton}-based extremum
  seeking,'' \emph{Automatica}, vol.~48, pp. 1759--1767, 2012.

\bibitem{K:2014}
M.~Krstic, ``Extremum seeking control,'' in \emph{Encyclopedia of Systems and
  Control}.\hskip 1em plus 0.5em minus 0.4em\relax Springer, 2014.

\bibitem{K:2002}
H.~K. Khalil, \emph{Nonlinear Systems}.\hskip 1em plus 0.5em minus 0.4em\relax
  Prentice Hall, 2002.

\bibitem{W:1971}
F.~Warner, \emph{Foundations of Differentiable Manifolds and Lie Groups}.\hskip
  1em plus 0.5em minus 0.4em\relax Scott Foresman and Company, 1971.

\bibitem{A:1957}
T.~Apostol, \emph{Mathematical Analysis - A Modern Approach to Advanced
  Calculus}.\hskip 1em plus 0.5em minus 0.4em\relax Addison-Wesley Publishing
  Company, 1957.

\end{thebibliography}
\end{document}